\documentclass[reqno,b5paper]{amsart}

\usepackage{amsmath}
\usepackage{amssymb}
\usepackage{amsthm}
\usepackage{enumerate}
\usepackage[mathscr]{eucal}
\usepackage{eqlist}
\usepackage[colorlinks,citecolor=red,pagebackref,hypertexnames=false]{hyperref}

\DeclareMathOperator{\supp}{supp}

\theoremstyle{plain}
\newtheorem{theorem}{Theorem}[section]

\newtheorem{lemma}{Lemma}[section]
\newtheorem{corol}{Corollary}[theorem]

\theoremstyle{definition}
\newtheorem{definition}{Definition}[section]
\newtheorem{remark}{\textnormal{\textbf{Remark}}}

\theoremstyle{remark}

\numberwithin{equation}{section}
\numberwithin{equation}{section}

\begin{document}

\title[The geometry of two-valued subsets of $L_{p}$-spaces]{The geometry of two-valued subsets of $L_{p}$-spaces}

\author{Anthony Weston}

\newcommand{\acr}{\newline\indent}

\address{Department of Mathematics and Statistics\acr
         Canisius College\acr
         Buffalo, NY 14208\acr
         USA}
\email{westona@canisius.edu}
\address{Department of Decision Sciences\acr
         University of South Africa\acr
         PO Box 392, UNISA 0003\acr
         SOUTH AFRICA}
\email{westoar@unisa.ac.za}

\subjclass[2010]{Primary 46B04; Secondary 46B85}
\keywords{$L_{p}$-space, (strict) negative type, isometry}

\begin{abstract}
Let $\mathcal{M}(\Omega, \mu)$ denote the algebra of all scalar-valued measurable functions
on a measure space $(\Omega, \mu)$. Let $B \subset \mathcal{M}(\Omega, \mu)$ be a set of finitely
supported measurable functions such that the essential range of each $f \in B$ is a subset of
$\{ 0,1 \}$. The main result of this paper shows that for any $p \in (0, \infty)$, $B$ has
strict $p$-negative type when viewed as a metric subspace of $L_{p}(\Omega, \mu)$ if and only
if $B$ is an affinely independent subset of $\mathcal{M}(\Omega, \mu)$ (when $\mathcal{M}(\Omega, \mu)$
is considered as a real vector space). It follows that every two-valued (Schauder) basis of
$L_{p}(\Omega, \mu)$ has strict $p$-negative type. For instance, for each $p \in (0, \infty)$,
the system of Walsh functions in $L_{p}[0,1]$ is seen to have strict $p$-negative type. The
techniques developed in this paper also provide a systematic way to construct, for any $p \in (2, \infty)$,
subsets of $L_{p}(\Omega, \mu)$ that have $p$-negative type but not $q$-negative type for any $q > p$.
Such sets preclude the existence of certain types of isometry into $L_{p}$-spaces.
\end{abstract}

\dedicatory{Dedicated to Nancy Joy Weston}

\maketitle

\section{Introduction}\label{sec:1}

A set $B \subset L_{p}(\Omega, \mu)$ is \textit{two-valued} if $|B| > 1$ and the
essential range of each $f \in B$ is a subset of $\{ 0,1 \}$. Our interest in two-valued
subsets of $L_{p}(\Omega, \mu)$ was piqued by the following elegant theorem.

\begin{theorem}[Murugan \cite{Mur}]\label{hyp:cube}
Suppose $k, n \geq 1$. A subset $B = \{ \mathbf{x}_{0}, \mathbf{x}_{1}, \ldots, \mathbf{x}_{k} \}$ of
the Hamming cube $\{ 0, 1\}^{n} \subset \ell_{1}^{(n)}$ is affinely independent if and only if
$B$ has strict $1$-negative type.
\end{theorem}

The two-valued subsets of $\ell_{1}^{(n)}$ are precisely those $B \subseteq \{ 0, 1\}^{n}$
that have at least two elements. A notable feature of Theorem \ref{hyp:cube} is that the subsets of
the Hamming cube that have strict $1$-negative type are characterized solely
in terms of the vector space structure of $\mathbb{R}^{n}$. In this paper we generalize
Theorem \ref{hyp:cube} to the setting of arbitrary two-valued subsets of $L_{p}$-spaces.
Our results are valid for all $p \in (0, \infty)$. In order to proceed it is helpful to recall
some basic information about classical and strict negative type.

Negative type was originally studied in relation to the problem of isometrically
embedding metric and normed spaces into $L_{p}$-spaces, $0 < p \leq 2$. Cayley \cite{Cay}, Menger \cite{Men},
Schoenberg \cite{Sc3} and Bretagnolle \textit{et al.}\ \cite{Bre} authored pivotal papers on this
embedding problem. The closely related notion of generalized roundness was developed by Enflo \cite{En2}
to study universal uniform embedding spaces. These notions may be defined in the following manner.

\begin{definition}\label{neg:gen} Let $p \geq 0$ and let $(X,d)$ be a metric space with $|X| > 1$.
\begin{enumerate}
\item $(X,d)$ has $p$-{\textit{negative type}} if and only if for all integers $n \geq 2$,
all finite subsets $\{z_{1}, \ldots , z_{n} \} \subseteq X$, and all scalar $n$-tuples
$\boldsymbol{\zeta} = (\zeta_{1}, \ldots, \zeta_{n}) \in \mathbb{R}^{n}$ that satisfy
$\zeta_{1} + \cdots + \zeta_{n} = 0$, we have
\begin{eqnarray}\label{p:neg}
\sum\limits_{j,i =1}^{n} d(z_{j},z_{i})^{p} \zeta_{j} \zeta_{i} & \leq & 0.
\end{eqnarray}

\item $(X,d)$ has \textit{strict} $p$-{\textit{negative type}} if and only if it has $p$-negative type
and the inequalities (\ref{p:neg}) are strict except in the trivial case $\boldsymbol{\zeta} = \boldsymbol{0}$.

\item The \textit{generalized roundness} of $(X,d)$, denoted by $\wp_{(X, d)}$ or simply $\wp_{X}$,
is the supremum of the set of all $p$ for which $(X,d)$ has $p$-negative type.
\end{enumerate}
\end{definition}

Most studies of negative type properties of subsets and subspaces of $L_{p}$-spaces have focussed
on the case $0 < p \leq 2$. There are a number of reasons why this restriction on the range of $p$
is observed in the literature. For instance, Schoenberg \cite{Sc3} noted that if $0 < p \leq 2$ and
if $(\Omega, \mu)$ is a measure space such that $L_{p}(\Omega, \mu)$ is at least two-dimensional,
then $L_{p}(\Omega, \mu)$ has $p$-negative type but it does not have $q$-negative type for any $q > p$.
This ensures that non-empty subsets of $L_{p}(\Omega, \mu)$ have at least $p$-negative type
provided $0 < p \leq 2$. The situation is vastly different if $p > 2$. In this case,
if $(\Omega, \mu)$ is a measure space such that $L_{p}(\Omega, \mu)$ is at least three-dimensional,
then $L_{p}(\Omega, \mu)$ does not have $q$-negative type for any $q > 0$. This implication follows
from theorems of Bretagnolle \textit{et al.}\ \cite{Bre}, Dor \cite{Dor}, Misiewicz \cite{Mis}
and Koldobsky \cite{Kol}. One effect of the latter result is that it makes it considerably more
difficult to determine which normed spaces embed linearly and isometrically into some $L_{p}$-space
when $p > 2$ (because $p$-negative type is no longer a necessary or sufficient condition).

In spite of this apparent roadblock when $p > 2$, it nevertheless seems to be an intriguing project
to determine subsets of $L_{p}(\Omega, \mu)$ that have $s$-negative type for some $s = s(p) > 0$
that depends upon $p$. This is one avenue of inquiry in this paper: for each $p \in (0, \infty)$ we
identify interesting subsets of $L_{p}(\Omega, \mu)$ that have $p$-negative type but not $q$-negative type for
any $q > p$. For the reasons we have indicated, this type of result is rare in the case $p > 2$.

For $p \in (0, 2)$, Kelleher \textit{et al}.\ \cite{Kel} have shown that if a subset $B$ of
$L_{p}(\Omega, \mu)$ is affinely independent (when $L_{p}(\Omega, \mu)$ is considered as a real vector space),
then $B$ has strict $p$-negative type. The converse of this statement is true when $p = 2$ but not when
$p < 2$ (see Theorem \ref{thm:inner} and Remark \ref{rem:ex}). However, in the case $p =1$, we see from Theorem
\ref{hyp:cube} that the converse statement can hold under additional assumptions. Murugan's proof of
Theorem \ref{hyp:cube} relies on novel properties of vertex transitive graphs and the special geometry
of the Hamming cube. Another way to prove Theorem \ref{hyp:cube} is to apply Nickolas and Wolf
\cite[Theorem 3.4 (3)]{Nic}. In a second avenue of inquiry in this paper, we obtain a significant generalization
of Murugan's theorem. It turns out that the only thing one really needs to know about the Hamming cube
is that it is a two-valued set in an $L_{p}$-space, $0 < p < \infty$. In fact, the two avenues of inquiry
that we have mentioned converge to one and the same thing: the geometry of two-valued sets in $L_{p}$-spaces.

The results in this paper are obtained by studying the nature of $p$-polygonal equalities in two-valued
subsets of $L_{p}$-spaces. Such equalities are completely characterized in terms of balanced simplices
in Theorem \ref{thm:1}. This is an interesting result because balanced simplices, being defined purely
in terms of the underlying vector space (Definition \ref{def:bal}), do not depend upon $p$. It follows,
at once, that each two-valued set $B \subset L_{p}(\Omega, \mu)$ has $p$-negative type. Furthermore, $B$
is seen to have strict $p$-negative type if and only if $B$ is affinely independent (when $L_{p}(\Omega, \mu)$
is considered as a real vector space). These results are valid for all $p \in (0, \infty)$ and they
give noteworthy information about the metric geometry of two-valued sets in $L_{p}$-spaces. For
instance, it follows that the system of Walsh functions in $L_{p}[0,1]$ has strict $p$-negative type
for all $p \in (0, \infty)$. We also see, for the first time in the literature, a systematic way to
construct subsets of $L_{p}(\Omega, \mu)$ that have generalized roundness $p$ for $p > 2$.

Throughout we assume that all measures are non-trivial and positive. We further assume that
$0 < p < \infty$ (unless stated otherwise) and that $(\Omega, \mu)$ is a measure space for
which $L_{p}(\Omega, \mu)$ is at least two-dimensional. If $p \in (0,1)$, then $L_{p}(\Omega, \mu)$ is
endowed with the usual quasi-norm. We also let $\mathcal{M}(\Omega, \mu)$ denote the vector space of
all scalar-valued measurable functions on $(\Omega, \mu)$. It is always the case that sums indexed
over the empty set are taken to be zero. When we refer to a subset of a metric space as having
$p$-negative type we mean, of course, the subset together with the metric that it inherits from the
ambient space.

\section{Signed simplices and polygonal equalities in $L_{p}$-spaces}\label{sec:2}

Signed simplices and $p$-polygonal equalities provide an effective alternative means for studying
classical and strict $p$-negative type.

\begin{definition}\label{S1}
Let $X$ be a set and suppose that $s,t > 0$ are integers.
A \textit{signed $(s,t)$-simplex in $X$} is a collection of (not necessarily distinct) points
$x_{1}, \ldots, x_{s}, y_{1}, \ldots, y_{t} \in X$ together with a corresponding collection
of real numbers $m_{1}, \ldots, m_{s}, n_{1}, \ldots, n_{t}$ that satisfy $m_{1} + \cdots + m_{s}
= n_{1} + \cdots + n_{t}$. Such a configuration of points and real numbers will be denoted by
$D = [x_{j}(m_{j});y_{i}(n_{i})]_{s,t}$ and will hereafter simply be called a \textit{simplex}.
\end{definition}

A simplex $D = [x_{j}(m_{j});y_{i}(n_{i})]_{s,t}$ in the real line will be called a
\textit{two-valued simplex} if $x_{j}, y_{i} \in \{ 0,1 \}$ for all $j, i$.
A simplex $D = [x_{j}(m_{j});y_{i}(n_{i})]_{s,t}$ in $L_{p}(\Omega, \mu)$ will be called a
\textit{two-valued simplex} if $\{ x_{j}, y_{i} \}$ is a two-valued set in $L_{p}(\Omega, \mu)$.

\begin{definition}\label{S3}
Given a signed $(s,t)$-simplex $D = [x_{j}(m_{j});y_{i}(n_{i})]_{s,t}$ in a set $X$ we
denote by $S(D)$ the set of distinct points in $X$ that appear in $D$.
For each $z \in S(D)$ we define the \textit{repeating numbers} ${\mathbf{m}}(z)$ and ${\mathbf{n}}(z)$ as follows:
\[
{\mathbf{m}}(z)  =  \sum\limits_{j : z = x_{j}} m_{j} \mbox{ and }
{\mathbf{n}}(z)  =  \sum\limits_{i : z = y_{i}} n_{i}.
\]
We say that the simplex $D$ is \textit{degenerate} if ${\mathbf{m}}(z) = {\mathbf{n}}(z)$ for all $z \in S(D)$.
\end{definition}

Informally, a simplex $D$ is degenerate if each $z \in S(D)$ is equally represented in both halves of the simplex.
Degenerate simplices in a metric space $(X, d)$ equate in a precise way to the trivial case $\boldsymbol{\zeta} =
\boldsymbol{0}$ in Definition \ref{neg:gen}.

\begin{definition}\label{S4}
Let $p \geq 0$ and let $(X,d)$ be a metric space.
For each signed $(s,t)$-simplex $D = [x_{j}(m_{j});y_{i}(n_{i})]_{s,t}$ in $X$ we define
\begin{eqnarray*}
\gamma_{p}(D) & = &
\sum\limits_{j,i = 1}^{s,t} m_{j}n_{i}d(x_{j},y_{i})^{p} -
\sum\limits_{ 1 \leq j_{1} < j_{2} \leq s} m_{j_{1}}m_{j_{2}}d(x_{j_{1}},x_{j_{2}})^{p} \\
              & ~ & - \sum\limits_{ 1 \leq i_{1} < i_{2} \leq t} n_{i_{1}}n_{i_{2}}d(y_{i_{1}},y_{i_{2}})^{p}.
\end{eqnarray*}
We call $\gamma_{p}(D)$ the \textit{$p$-simplex gap of $D$ in $(X,d)$}.
\end{definition}

The fundamental relationships between negative type and simplex gaps are expressed in the following theorem.

\begin{theorem}[Kelleher \textit{et al}.\ \cite{Kel}]\label{2.3}
Let $p > 0$ and let $(X,d)$ be a metric space. Then the following conditions are equivalent:
\begin{enumerate}
\item $(X,d)$ has $p$-negative type.

\item $\gamma_{p}(D) \geq 0$ for each signed simplex $D$ in $X$.
\end{enumerate}
Moreover, $(X,d)$ has strict $p$-negative type if and only if $\gamma_{p}(D) > 0$ for each
non-degenerate signed simplex $D$ in $(X,d)$.
\end{theorem}

Theorem \ref{2.3} is a variation of themes developed in Weston \cite{We1}, Lennard \textit{et al}.\
\cite{Ltw}, and Doust and Weston \cite{Dou}. The motivation for all such studies has stemmed from
Enflo's original definition of generalized roundness that was given in \cite{En2}. Theorem \ref{2.3}
motivates the notion of a (non-trivial) $p$-polygonal equality. The nature of such equalities in
two-valued subsets of $L_{p}(\Omega, \mu)$ will be central to the development of our main results in
Section \ref{sec:3}.

\begin{definition}\label{S6}
Let $p \geq 0$ and let $(X,d)$ be a metric space.
A \textit{$p$-polygonal equality} in $(X, d)$ is an equality of the form $\gamma_{p}(D) = 0$
where $D$ is a simplex in $X$. If, moreover, the simplex $D$ is non-degenerate, we will say that
the $p$-polygonal equality $\gamma_{p}(D) = 0$ is \textit{non-trivial}.
\end{definition}

For $p \in (0, 2)$, the non-empty subsets of $L_{p}(\Omega, \mu)$ that have strict $p$-negative
type may be characterized in terms of so-called virtually degenerate simplices.

\begin{definition}\label{v:degenerate}
A non-degenerate simplex $D = [x_{j}(m_{j});y_{i}(n_{i})]_{s,t}$ in $L_{p}(\Omega, \mu)$ is said to be
\textit{virtually degenerate} if the family of evaluation simplices $D(\omega) = [x_{j}(\omega)(m_{j});y_{i}(\omega)(n_{i})]_{s,t}$,
$\omega \in \Omega$, are degenerate in the scalar field of $L_{p}(\Omega, \mu)$ $\mu$-almost everywhere.
\end{definition}

It is not known to what extent the following theorem may be generalized to the case $p \in (2, \infty)$.

\begin{theorem}[Kelleher \textit{et al}.\ \cite{Kel}]\label{Lp}
Suppose $p \in (0,2)$. A non-empty subset of $L_{p}(\Omega, \mu)$ has strict $p$-negative type if
and only if it does not admit any virtually degenerate simplices.
\end{theorem}

A difficulty with Theorem \ref{Lp} is that the description of the subsets of $L_{p}(\Omega, \mu)$ that
have strict $p$-negative type is not purely geometric. In practice, it is a hard problem to decide
if a set $B \subset L_{p}(\Omega, \mu)$ admits a virtually degenerate simplex. However, virtually
degenerate simplices satisfy the following definition.

\begin{definition}\label{def:bal}
Let $D = [x_{j}(m_{j});y_{i}(n_{i})]_{s,t}$ be a simplex in a vector space $X$.
We say that $D$ is \textit{balanced} if $\sum m_{j}x_{j} = \sum n_{i}y_{i}$.
\end{definition}

Informally, a simplex $D$ in a vector space $X$ is balanced if the two halves of the simplex have
the same center of gravity. In a real or complex vector space there is a direct link between non-degenerate
balanced simplices and affinely dependent subsets. This theorem underpins the main considerations of this paper.

\begin{theorem}[Kelleher \textit{et al}.\ \cite{Kel}]\label{thm:B}
Let $n \geq 1$ be an integer and let $X$ be a real or complex vector space. Then a subset
$Z = \{ z_{0}, z_{1}, \ldots z_{n} \}$ of $X$ admits a non-degenerate balanced simplex if and
only if the set $Z$ is affinely dependent (when $X$ is considered as a real vector space).
\end{theorem}

Theorem \ref{thm:B} leads to a complete description of the subsets of inner product spaces
that have strict $2$-negative type.

\begin{theorem}[Kelleher \textit{et al}.\ \cite{Kel}]\label{thm:inner}
A non-empty subset $Z$ of a real or complex inner product space $X$ has strict $2$-negative type
if and only if $Z$ is an affinely independent subset of $X$ (when $X$ is considered as a real vector space).
\end{theorem}

A notable feature of Theorem \ref{thm:inner} is that the vector space structure of an inner
product space $X$ completely determines the subsets of $X$ that have strict $2$-negative type.
In the next section we will encounter a similar phenomenon for two-valued subsets of
$L_{p}$-spaces, $0 < p < \infty$.

\section{The geometry of two-valued subsets of $L_{p}(\Omega, \mu)$}\label{sec:3}

We begin this section by classifying the (non-trivial) $p$-polygonal equalities in two-valued subsets
of $L_{p}$-spaces, $0 < p < \infty$.

\begin{lemma}\label{lemma:1}
Let $p \in (0, \infty)$ be given.
If $D = [x_{j}(m_{j});y_{i}(n_{i})]_{s,t}$ is a two-valued simplex in the real line, then
\[
\gamma_{p}(D) = \gamma_{1}(D) = \biggl| \sum\limits_{j} m_{j}x_{j} - \sum\limits_{i} n_{i}y_{i} \biggl|^{2}.
\]
\end{lemma}

\begin{proof} Let $p > 0$ be given.
Suppose that $D = [x_{j}(m_{j});y_{i}(n_{i})]_{s,t}$ is a two-valued simplex in the real line.
By definition, we have $x_{j}, y_{i} \in \{ 0, 1 \}$ for all $j$ and $i$. Thus all distances
appearing in the definition of $\gamma_{p}(D)$ are $0$ or $1$. As a result, we see that $\gamma_{p}(D) =
\gamma_{1}(D)$. This establishes the first equality of the lemma.

To establish the second equality, we create a new simplex by ``compressing'' the original simplex $D$.
We do this by setting $x_{1}^{\ast} = y_{1}^{\ast} = 0$ and $x_{2}^{\ast} = y_{2}^{\ast} = 1$. Weights
for these vertices are then given by the corresponding repeating numbers from the original simplex $D$:
$m_{1}^{\ast} = {\mathbf{m}}(0), m_{2}^{\ast} = {\mathbf{m}}(1), n_{1}^{\ast} = {\mathbf{n}}(0)$,
and $n_{2}^{\ast} = {\mathbf{n}}(1)$.

Notice that $m_{1}^{\ast} + m_{2}^{\ast} = m_{1} + \cdots m_{s} = n_{1} + \cdots n_{t} = n_{1}^{\ast} + n_{2}^{\ast}$.
The compressed simplex is then simply defined to be $D^{\ast} = [x_{j}^{\ast}(m_{j}^{\ast});y_{i}^{\ast}(n_{i}^{\ast})]_{2,2}$.
It easy to verify that $\gamma_{p}(D) = \gamma_{1}(D) = \gamma_{1}(D^{\ast})$.

By definition of $D$ and $D^{\ast}$,
\begin{eqnarray}\label{Leq:1}
\biggl| \sum\limits_{j} m_{j}x_{j} - \sum\limits_{i} n_{i}y_{i} \biggl|^{2} & = &
\biggl|\sum\limits_{j} m_{j}^{\ast}x_{j}^{\ast} - \sum\limits_{i} n_{i}^{\ast}y_{i}^{\ast} \biggl|^{2} \nonumber \\
& = & (m_{2}^{\ast} - n_{2}^{\ast})^{2}.
\end{eqnarray}

On the other hand,
\begin{eqnarray}\label{Leq:2}
\gamma_{p}(D) & = & \gamma_{1}(D^{\ast}) \nonumber \\
              & = & m_{1}^{\ast}n_{2}^{\ast} + m_{2}^{\ast}n_{1}^{\ast} - m_{1}^{\ast}m_{2}^{\ast} - n_{1}^{\ast}n_{2}^{\ast} \nonumber \\
              & = & (n_{1}^{\ast} - m_{1}^{\ast})(m_{2}^{\ast} - n_{2}^{\ast}) \nonumber \\
              & = & (m_{2}^{\ast} - n_{2}^{\ast})^{2}.
\end{eqnarray}
The lemma now follows from (\ref{Leq:1}) and (\ref{Leq:2}).
\end{proof}

\begin{theorem}\label{thm:1}
Let $p \in (0, \infty)$ be given.
If $D = [x_{j}(m_{j});y_{i}(n_{i})]_{s,t}$ is a two-valued simplex in $L_{p}(\Omega, \mu)$, then
\begin{eqnarray}\label{Teq:1}
\gamma_{p}(D) & = & \biggl\| \sum\limits_{j} m_{j}x_{j} - \sum\limits_{i} n_{i}y_{i} \biggl\|_{2}^{2}.
\end{eqnarray}
In particular, $\gamma_{p}(D) = 0$ if and only if the simplex $D$ is balanced.
\end{theorem}

\begin{proof}
Let $p > 0$ be given. Suppose that $D = [x_{j}(m_{j});y_{i}(n_{i})]_{s,t}$ is a two-valued simplex in $L_{p}(\Omega, \mu)$.
Except on a set of $\mu$-measure zero,
the evaluation simplices $D(\omega) = [x_{j}(\omega)(m_{j});y_{i}(\omega)(n_{i})]_{s,t}$, $\omega \in \Omega$, are
two-valued in the real line. So, by Lemma \ref{lemma:1}, we have
\begin{eqnarray*}
\gamma_{p}(D(\omega)) & = &
\biggl| \sum\limits_{j} m_{j}x_{j}(\omega) - \sum\limits_{i} n_{i}y_{i}(\omega) \biggl|^{2},
\end{eqnarray*}
$\mu$-almost everywhere. Integrating over $\Omega$ with respect to $\mu$ gives the desired conclusion.
\end{proof}

It is an immediate consequence of Theorem \ref{thm:1} that two-valued subsets of $L_{p}(\Omega, \mu)$
have $p$-negative type, $0 < p < \infty$. In fact, the following considerably stronger statement applies.

\begin{corol}\label{cor:1}
Let $B \subset L_{p}(\Omega, \mu)$ be a two-valued set. Then:
\begin{enumerate}
\item[(1)] $B$ has $p$-negative type.

\item[(2)] $B$ has strict $p$-negative type if and only if
$B$ is affinely independent (when $L_{p}(\Omega, \mu)$ is considered as a real vector space).

\item[(3)] If $B$ is affinely dependent (when $L_{p}(\Omega, \mu)$ is considered as a real vector space),
then $B$ has generalized roundness $p$.
\end{enumerate}
\end{corol}

\begin{proof}
For any simplex $D = [x_{j}(m_{j});y_{i}(n_{i})]_{s,t}$ in $B$ we have $\gamma_{p}(D) \geq 0$ by (\ref{Teq:1}).
Thus $B$ has $p$-negative type by part (1) of Theorem \ref{2.3}.

To prove (2) we argue contrapositively. Suppose that $B$ does not have strict $p$-negative type. Then,
by part (2) of Theorem \ref{2.3}, there must be a non-degenerate simplex $D$ in $B$ such that $\gamma_{p}(D) = 0$.
Hence $D$ is balanced by the second statement of Theorem \ref{thm:1}. Since $D$ is non-degenerate and balanced
it follows from Theorem \ref{thm:B} that $S(D)$ is affinely dependent (when $L_{p}(\Omega, \mu)$ is
considered as a real vector space). As $S(D) \subseteq B$, this shows that $B$ is affinely dependent
(when $L_{p}(\Omega, \mu)$ is considered as a real vector space).

Now suppose that $B$ is affinely dependent (when $L_{p}(\Omega, \mu)$ is considered as a real vector space).
Then $B$ admits a non-degenerate balanced simplex $D$ by Theorem \ref{thm:B}. Thus $\gamma_{p}(D) = 0$ by
the second statement of Theorem \ref{thm:1}. As $D$ is also non-degenerate we conclude from part (2) of Theorem \ref{2.3}
that $S(D)$, and hence $B$, does not have strict $p$-negative type. Moreover, since $B$ does not have
strict $p$-negative type, it cannot have $q$-negative type for any $q > p$ by Li and Weston \cite[Theorem 5.4]{Hli}.
However, $B$ has $p$-negative type by (1). Thus $\wp_{B} = p$. This completes the proof of (2) and (3).
\end{proof}

\begin{remark}\label{rem:ex}
For $p \not= 2$ the forward implication of Corollary \ref{cor:1} (2) does not necessarily hold for
subsets of $L_{p}(\Omega, \mu)$ that take three or more values. For example, consider the points
$z_{0} = (0,0), z_{1} = (1,1), z_{2} = (2,1)$ and $z_{3} = (2,0) \in \ell_{p}^{(2)}$ with $0 < p < 2$.
The set $Z$ takes the values $\{ 0, 1, 2 \}$ and is affinely dependent. Moreover, it is easy to see that
$Z$ does not admit any virtually degenerate simplices. Thus $Z$ has strict $p$-negative type by
Theorem \ref{Lp}. An even simpler example can be constructed in the case $p > 2$ by considering
any three-valued set in $\ell_{p}^{(2)}$ that contains distinct points $x, y$ and $z$ such that $z = (x + y)/2$.
The set $\{ x,y,z \} \subset \ell_{p}^{(2)}$ does not have $q$-negative type for any $q > 2$ because it
contains a metric midpoint.
\end{remark}

If $0 < p \leq 2$, Corollary \ref{cor:1} (1) holds for any subset of $L_{p}(\Omega, \mu)$. However,
if $2 < p < \infty$ and $L_{p}(\Omega, \mu)$ is at least three-dimensional, then $L_{p}(\Omega, \mu)$
does not have $q$-negative type for any $q > 0$. In this instance, Corollary \ref{cor:1} (1) is quite
striking since it identifies subsets of $L_{p}(\Omega, \mu)$ that have, at least, $p$-negative type.

The next corollary is an immediate consequence of Corollary \ref{cor:1} (2) because (Schauder) bases
are linearly independent.

\begin{corol}\label{cor:1.5}
Every two-valued (Schauder) basis of $L_{p}(\Omega, \mu)$ has strict $p$-negative type.
\end{corol}

It is also the case that Corollary \ref{cor:1} (2) prohibits the existence of certain types of isometry.

\begin{corol}\label{cor:2}
Let $B$ be any two-valued subset of $L_{p}(\Omega, \mu)$ that is affinely dependent (when
$L_{p}(\Omega, \mu)$ is considered as a real vector space). Then:
\begin{enumerate}
\item[(1)] No metric space that has strict $p$-negative type is isometric to $B$. In particular, no
ultrametric space is isometric to $B$.

\item[(2)] If $0 < p < q \leq 2$, no subset of any $L_{q}$-space is isometric to $B$.
\end{enumerate}
\end{corol}

\begin{proof}
$B$ does not have strict $p$-negative type by Corollary \ref{cor:1} (2). So no metric space
that has strict $p$-negative type can be isometric to $B$. The second statement of part (1)
is then immediate because ultrametric spaces have strict $q$-negative type for all $q > 0$
by Faver \textit{et al}.\ \cite[Corollary 5.3]{Fav}.

Now suppose that $0 < p < 2$. If $A$ is a subset of an $L_{q}$-space such that $|A| > 1$
and $q \in (p, 2]$, then $A$ has $q$-negative type. Thus $A$ has strict $p$-negative type
by Li and Weston \cite[Theorem 5.4]{Hli}. Hence $A$ is not isometric to $B$ by part (1). 
\end{proof}

It is possible, and worthwhile, to recast Corollary \ref{cor:1} (2) in terms of $\mathcal{M}(\Omega, \mu)$.
Recall that $\mathcal{M}(\Omega, \mu)$ denotes the vector space of all scalar-valued measurable functions
on $(\Omega, \mu)$. We will say that a function $f \in \mathcal{M}(\Omega, \mu)$ is
\textit{finitely supported} if $\mu (\supp (f)) < \infty$, and that a set $B \subset \mathcal{M}(\Omega, \mu)$
is \textit{finitely supported} if each $f \in B$ is finitely supported.

At the beginning of Section \ref{sec:1} we stated that a set $B \subset L_{p}(\Omega, \mu)$ is
\textit{two-valued} if $|B| > 1$ and the essential range of each $f \in B$ is a subset of $\{ 0,1 \}$.
More generally, we will say that a set $B \subset \mathcal{M}(\Omega, \mu)$ is
\textit{two-valued} if $|B| > 1$ and the essential range of each $f \in B$ is a subset of $\{ 0,1 \}$.
Clearly, the essential range of a function $f \in \mathcal{M}(\Omega, \mu)$ is a subset of $\{ 0,1 \}$
if and only if $f = \chi_{\supp (f)}$ $\mu$-almost everywhere. Consequently, for any $p \in (0, \infty)$
and any function $f \in \mathcal{M}(\Omega, \mu)$ whose essential range is a subset of $\{ 0,1 \}$,
we see that $f \in L_{p}(\Omega, \mu)$ if and only if $f$ is finitely supported. This observation implies
the following useful lemma.

\begin{lemma}\label{lem:2}
Let $B$ be a two-valued set in $\mathcal{M}(\Omega, \mu)$.
Then, for any $p \in (0, \infty)$, the following statements are equivalent:
\begin{enumerate}
\item $B$ is finitely supported.

\item $B$ is a two-valued set in $L_{p}(\Omega, \mu)$.
\end{enumerate}
\end{lemma}

A key feature of Corollary \ref{cor:1} (2) is the strong interplay that it illustrates between
the metric geometry of $L_{p}(\Omega, \mu)$ and the underlying vector space. The property of strict
$p$-negative type depends explicitly upon $p$, whereas the property of affine independence depends
only on the underlying vector space. However, for any $p \in (0, \infty)$, the underlying vector
space is a vector subspace of $\mathcal{M}(\Omega, \mu)$. This is important because $\mathcal{M}(\Omega, \mu)$
is completely independent of $p$. Thus we obtain a generalization of Theorem \ref{hyp:cube}
that is valid for any $p \in (0, \infty)$.

\begin{corol}\label{cor:3}
Let $B$ be a finitely supported two-valued set in $\mathcal{M}(\Omega, \mu)$.
Then, for any $p \in (0, \infty)$, the following statements are equivalent:
\begin{enumerate}
\item $B$ has strict $p$-negative type when viewed as a metric subspace of $L_{p}(\Omega, \mu)$.

\item $B$ is an affinely independent subset of $\mathcal{M}(\Omega, \mu)$
(when $\mathcal{M}(\Omega, \mu)$ is considered as a real vector space).
\end{enumerate}
\end{corol}

\begin{proof}
Immediate from Corollary \ref{cor:1} and Lemma \ref{lem:2}.
\end{proof}

Corollary \ref{cor:3} has interesting outcomes in the finite-dimensional setting of (real) $\ell_{p}^{(n)}$.
In this case, $\mathcal{M}(\Omega, \mu) = \mathbb{R}^{n}$. Letting $d_{p}$ denote the metric on $\mathbb{R}^{n}$
induced by the $p$-norm we deduce the following consequences of Corollary \ref{cor:3}.

\begin{corol}\label{cor:4}
Let $B$ be any two-valued set in $\mathbb{R}^{n}$. If $|B| > n + 1$, then the generalized
roundness of the metric space $(B, d_{p})$ is $p$ for all $p \in (0, \infty)$.
\end{corol}

\begin{proof}
Let $B = \{ \mathbf{x}_{0}, \mathbf{x}_{1}, \ldots, \mathbf{x}_{k} \}$ be a given two-valued set in
$\mathbb{R}^{n}$ such that $k > n$. Let $p \in (0, \infty)$ be given. By Corollary \ref{cor:1} (1),
$(B, d_{p})$ has $p$-negative type. However, the vectors
$\{ \mathbf{x}_{1} - \mathbf{x}_{0}, \mathbf{x}_{2} - \mathbf{x}_{0}, \ldots, \mathbf{x}_{k} - \mathbf{x}_{0} \}$
are linearly dependent because $k > n$. So it follows from Corollary \ref{cor:3} that
$(B, d_{p})$ does not have strict $p$-negative type. This precludes $(B, d_{p})$ from having $q$-negative type
for any $q > p$ by Li and Weston \cite[Theorem 5.4]{Hli}. Thus $\wp_{(B, d_{p})} = p$, as asserted.
\end{proof}

\begin{corol}\label{cor:5}
Let $B$ be any two-valued set in $\mathbb{R}^{n}$ and let $(X, d)$ be a $k$-point metric space
that has strict $p$-negative type for some $p \in (0, \infty)$. If $(X, d)$ embeds isometrically
into $(B, d_{p})$, then $n \geq k - 1$.
\end{corol}

\begin{proof}
Let $B$ be a given two-valued set in $\mathbb{R}^{n}$ and let $(X, d)$ be a $k$-point metric
space that has strict $p$-negative type for some $p \in (0, \infty)$. If $(X, d)$ isometrically
embeds into $(B, d_{p})$, then there is a (necessarily two-valued) subset $\tilde{X}$ of $B$
such that $(\tilde{X}, d_{p})$ has strict $p$-negative type. Thus $\tilde{X}$ is an affinely
independent subset of $\mathbb{R}^{n}$ by Corollary \ref{cor:3}. Therefore $k = |\tilde{X}| \leq n + 1$.
\end{proof}

We remark that Corollary \ref{cor:4} and Corollary \ref{cor:5} generalize similar results for the
Hamming cube $\{ 0,1 \}^{n} \subset \ell_{1}^{(n)}$ given in Murugan \cite{Mur}.

\begin{remark}
It is worth noting that the results of this section hold for any $\{ \alpha, \beta \}$-valued
subset $B$ of $L_{p}(\Omega, \mu)$, $\alpha \not= \beta$. Indeed, we may assume that $\alpha = 0$
by a translation. Then, given any $p > 0$ and any $\{ 0, \beta\}$-valued simplex $D = [x_{j}(m_{j});y_{i}(n_{i})]_{s,t}$
in $L_{p}(\Omega, \mu)$, it follows that
\begin{eqnarray*}
\gamma_{p}(D) & = & |\beta|^{p-2} \cdot \biggl\| \sum\limits_{j} m_{j}x_{j} - \sum\limits_{i} n_{i}y_{i} \biggl\|_{2}^{2}
\end{eqnarray*}
by slightly modifying the statement and proof of Lemma \ref{lemma:1}. So, for any $p \in (0, \infty)$, we see that
$\gamma_{p}(D) = 0$ if and only if the simplex $D$ is balanced. For example, by Corollary \ref{cor:1}, it then
follows that the classical system $B$ of Walsh functions in $L_{p}[0,1]$ forms a set of strict $p$-negative type.
\end{remark}

We conclude this paper by commenting briefly on the case $p = \infty$. If $B$ is a two-valued subset of
$\mathcal{M}(\Omega, \mu)$, then it is automatically a subset of $L_{\infty}(\Omega, \mu)$.
Moreover, as a metric subspace of $L_{\infty}(\Omega, \mu)$, $B$ is an ultrametric space.
Therefore $B$ has strict $p$-negative type for all $p \in (0, \infty)$ by Faver \textit{et al}.\
\cite[Corollary 5.3]{Fav}.

\section*{Acknowledgments}

I am indebted to the referees for their thoughtful comments on the preliminary version of this paper.
The outcome is a far more general and appealing paper. I would also like to thank the CCRDS
research group in the Department of Decision Sciences at the University of South Africa for their
kind support and insightful input during the preparation of this paper.

\end{document}